\documentclass[12pt]{amsart}  

\usepackage{amsmath,amsthm,amssymb,amscd,mathdots,enumerate,shuffle,hyperref}
\usepackage{setspace,mathrsfs}
\usepackage{arydshln}
\usepackage{tikz}
\usetikzlibrary{decorations.pathreplacing,angles,quotes}
\usetikzlibrary{calc}
\usetikzlibrary{arrows}

\allowdisplaybreaks

\newtheorem{theorem}{Theorem}[section]

\newtheorem{definition}[theorem]{Definition}
\newtheorem{lemma}[theorem]{Lemma}

\newtheorem{remark}[theorem]{Remark}

\numberwithin{equation}{section}

\newcommand{\HH}{ \mathcal{H}}
\newcommand{\HM}{ \widetilde{\mathcal{H}}}
\newcommand{\h}{\mathfrak H}
\newcommand{\FF}{ \mathcal{F}}
\newcommand{\FM}{ \widetilde{\mathcal{F}}}
\newcommand{\Q}{\mathbb{Q}}
\newcommand{\Z}{\mathbb{Z}}
\newcommand{\R}{\mathbb{R}}
\newcommand{\I}{\mathbb{I}}
\newcommand{\id}{\operatorname{id}}
\newcommand{\End}{\operatorname{End}}
\newcommand{\Aut}{\operatorname{Aut}}
\newcommand{\W}{\Theta}
\newcommand{\E}{\mathcal{E}}

\newcommand{\w}{w}
\newcommand{\f}{f}
\title{Rooted tree maps and the Kawashima relations for multiple zeta values}

\author{Henrik Bachmann}
\address{Graduate School of Mathematics,  Nagoya University, Nagoya, Japan.}
\email{henrik.bachmann@math.nagoya-u.ac.jp}

\author{Tatsushi Tanaka}
\address{Department of Mathematics, Kyoto Sangyo University, Kyoto, Japan.}
\email{t.tanaka@cc.kyoto-su.ac.jp}

\subjclass[2010]{05C05, 05C25, 11M32, 16T05}
\keywords{Hopf algebra of rooted trees, noncommutative polynomial algebra, multiple zeta values, Kawashima relations}

\begin{document}
\maketitle

\begin{abstract} 
Recently, inspired by the Connes-Kreimer Hopf algebra of rooted trees, the second named author introduced rooted tree maps as a family of linear maps on the noncommutative polynomial algebra in two letters. These give a class of relations among multiple zeta values, which are known to be a subclass of the so-called linear part of the Kawashima relations. In this paper we show the opposite implication, that is the linear part of the Kawashima relations is implied by the relations coming from rooted tree maps. 
\end{abstract}

\section{Introduction}
Rooted tree maps were introduced in \cite{T} and they assign to a rooted tree a linear map on the space $\h = \Q\langle x,y \rangle$ of noncommutative polynomials in $x$ and $y$. One application of these maps is that any admissible word evaluated at a rooted tree map gives a $\Q$-linear relation between multiple zeta values. To prove this result, one shows that these type of relations follow from a special case, the linear part, of the so called Kawashima relations.

The purpose of this note is to show, that the rooted tree maps relations are actually equivalent to the linear part of the Kawashima relations. 

For $k_1\geq 2, k_2,\dots,k_r \geq 1$  the multiple zeta values are defined by
\begin{equation} \label{eq:defmzv}
\zeta(k_1,\dots,k_r) = \sum_{m_1 > \dots > m_r > 0} \frac{1}{m_1^{k_1} \dots m_r^{k_r}} \,.
\end{equation}
To state the Kawashima relations we need to introduce some notations. Denote by $\h^1 = \Q + \h y$ the subspace of words in $\h$, which end in $y$. The space $\h^1$ is spanned by the monomials $z_{k_1}\dots z_{k_r}$ with $k_1,\dots,k_r \geq 1$, where $z_k=x^{k-1}y$. On $\h^1$ one can define recursively the harmonic product $\ast$ by $1 \ast w = w \ast 1 = w$ and
\[z_{k_1} v \ast z_{k_2} w = z_{k_1} (v \ast z_{k_2} w) + z_{k_2} (z_{k_1} v \ast w) + z_{k_1+k_2} (v \ast w) \]
for $k_1,k_2 \geq 1$ and $v , w \in \h^1$. 
Let $\h^0 = \Q + x\h y \subset \h$ be the subspace of admissible words and define the $\Q$-linear map $Z: \h^0 \rightarrow \R$ on a monomial $w=z_{k_1}\dots z_{k_r}$ by $Z(w) = \zeta(k_1,\dots,k_r)$. Equipped with the harmonic product $\ast$ both $\h^1$ and $\h^0$ are commutative $\Q$-algebras  and it is a well-known fact, that $Z$ is an algebra homomorphism from $\h^0$ to the algebra of multiple zeta values. Define the automorphism $\varphi \in \Aut(\h)$ (with respect to the concatenation) on the generators by $\varphi(x) =  x+y$ and $\varphi(y) = -y$ and define for words $v,w \in \h^1$ the operator $z_p v \circledast z_q w = z_{p+q} (v \ast w)$. With this the Kawashima relations can be stated as follows:
\begin{theorem}(\cite[Corollary 5.4]{K})\label{thm:kawashimarel} For all $v,w \in \h y$ and $m\geq 1$ we have
\begin{equation}
\sum_{\substack{i+j=m\\i,j \geq 1}} Z(\varphi(v) \circledast y^i) Z(\varphi(w) \circledast y^j)  = Z( \varphi(v \ast w) \circledast y^m) \,.
\end{equation}
\end{theorem}
It is expected that Theorem \ref{thm:kawashimarel} gives all $\Q$-linear relations between multiple zeta values after evaluating the product on the left-hand side by the shuffle product formula, whose definition we will omit here. Moreover numerical experiment suggest, that the two cases $m=1,2$ are enough to obtain all linear relations. 

Rooted tree maps assign to a rooted tree a linear map on the space $\h$. Evaluated at any admissible word these give also relations between multiple zeta values. Denoting by $\HM_{d\geq 1} \subset \End(\h)$ the space of all rooted tree maps of non-zero degree (see Section \ref{sec:rtm} for precise definitions), we have the following result proven in \cite{T}.
\begin{theorem}(\cite[Theorem 1.3]{T})\label{thm:rtmrelation} For any rooted tree map $f \in \HM_{d\geq 1}$ we have $$f(\h^0) \subset \ker Z.$$
\end{theorem}

The main result of this work is the following. 
\begin{theorem} \label{thm:main10}The rooted tree maps relations are equivalent to the linear part of the Kawashima relations, i.e. Theorem \ref{thm:rtmrelation} implies the $m=1$ case of Theorem \ref{thm:kawashimarel}. 
\end{theorem}

Let $R_k$ be the number of linearly independent rooted tree maps relations (Theorem \ref{thm:rtmrelation}) among multiple zeta values in weight $k$ (the number $k_1+\dots+k_r$ in \eqref{eq:defmzv}) and denote by $C_k$ the conjectured number of all linearly independent relations of multiple zeta values in weight $k$. For small $k$ these are given by
\bgroup
\def\arraystretch{1.3}%
\begin{table}[h!]
\centering
\begin{tabular}{c|c|c|c|c|c|c|c|c|c|c|c|c|c|c}
$k$          &   2 & 3 & 4 & 5 & 6 & 7 & 8 & 9 & 10 & 11 & 12 & 13 \\ \hline
$R_k$ &  0 & 1 & 2 & 5 & 10& 23 & 46 & 98 & 200 & 413 & 838 & 1713 \\ 
$C_k$ &  0 & 1 & 3 & 6 & 14& 29 & 60 & 123 & 249 & 503 & 1012 & 2032  
\end{tabular}
\end{table}
\egroup 


\subsection*{Acknowledgment}
The authors would like to thank the Max-Planck-Institut f\"ur Mathematik and the Hausdorff Research Institute for Mathematics  in Bonn for hospitality and support. The second author was also supported by Kyoto Sangyo University Research Grants.

\section{Rooted tree maps}\label{sec:rtm}
A rooted tree is as a finite graph which is connected, has no
cycles, and has a distinguished vertex called the root. We draw rooted trees with the root on top and we just consider rooted trees with no plane structure, which means that we for example do not distinguish between  $\begin{tikzpicture}[scale=0.25,baseline={([yshift=-.5ex]current bounding box.center)}]
\def\cz{5}
\def\wi{0.5}

\newcommand{\ci}[1]{	
	\fill[black] (#1) circle (\cz pt);
	\draw (#1) circle (\cz pt);
}

\coordinate (R) at (0,0);

\coordinate (r1) at (\wi,-1);
\coordinate (l1) at (-\wi,-1);

\coordinate (l11) at (-\wi-\wi,-2);
\coordinate (l12) at (-\wi+\wi,-2);

\draw (R) to (r1);
\draw (R) to (l1);

\draw (l1) to (l11);
\draw (l1) to (l12);

\ci{R}
\ci{r1}
\ci{l1}
\ci{l11}
\ci{l12}
\end{tikzpicture}$ and $\begin{tikzpicture}[scale=0.25,baseline={([yshift=-.5ex]current bounding box.center)}]
\def\cz{5}
\def\wi{0.5}

\newcommand{\ci}[1]{	
	\fill[black] (#1) circle (\cz pt);
	\draw (#1) circle (\cz pt);
}

\coordinate (R) at (0,0);

\coordinate (r1) at (\wi,-1);
\coordinate (l1) at (-\wi,-1);

\coordinate (r11) at (\wi-\wi,-2);
\coordinate (r12) at (\wi+\wi,-2);

\draw (R) to (r1);
\draw (R) to (l1);

\draw (r1) to (r11);
\draw (r1) to (r12);

\ci{R}
\ci{r1}
\ci{l1}
\ci{r11}
\ci{r12}
\end{tikzpicture}$. A product (given by the disjoint union) of rooted trees will be called a (rooted) forest and by $\HH$ we denote the $\Q$-algebra of forests generated by all trees. The unit of $\HH$, given by the empty forest, will be denoted by $\I$. Since we just consider trees without plane structure the algebra $\HH$ is commutative.  Due to the work of Connes and Kreimer (\cite{CK}) the space $\HH$ has the structure of a Hopf algebra. To define the coproduct on $\HH$ we first define the linear map $B_+$ on $\HH$, which connects all roots of the trees in a forest to a new root. For example we have $B_+\left(\begin{tikzpicture}[scale=0.3,baseline={([yshift=-.5ex]current bounding box.center)}]
\def\cz{5}
\def\wi{0.5}

\newcommand{\ci}[1]{	
	\fill[black] (#1) circle (\cz pt);
	\draw (#1) circle (\cz pt);
}

\coordinate (R) at (0,0);
\coordinate (r1) at (\wi,-1);
\coordinate (l1) at (-\wi,-1);

\draw (R) to (l1);
\draw (R) to (r1);

\ci{R};
\ci{l1};
\ci{r1};
\end{tikzpicture} \,\,\begin{tikzpicture}[scale=0.3,baseline={([yshift=-.5ex]current bounding box.center)}]
\def\cz{5}
\def\wi{0.5}

\newcommand{\ci}[1]{	
	\fill[black] (#1) circle (\cz pt);
	\draw (#1) circle (\cz pt);
}

\coordinate (R) at (0,0);
\ci{R}
\end{tikzpicture}\right) = \begin{tikzpicture}[scale=0.25,baseline={([yshift=-.5ex]current bounding box.center)}]
\def\cz{5}
\def\wi{0.5}

\newcommand{\ci}[1]{	
	\fill[black] (#1) circle (\cz pt);
	\draw (#1) circle (\cz pt);
}

\coordinate (R) at (0,0);

\coordinate (r1) at (\wi,-1);
\coordinate (l1) at (-\wi,-1);

\coordinate (l11) at (-\wi-\wi,-2);
\coordinate (l12) at (-\wi+\wi,-2);

\draw (R) to (r1);
\draw (R) to (l1);

\draw (l1) to (l11);
\draw (l1) to (l12);

\ci{R}
\ci{r1}
\ci{l1}
\ci{l11}
\ci{l12}
\end{tikzpicture}$. Clearly for every tree $t \in \HH$ there exists a unique forest $f_t \in \HH$ with $t = B_+(f_t)$, which is just given by removing the root of $t$.
The coproduct on $\HH$ can then be defined recursively for a tree $t \in \HH$ by
\[ \Delta(t) = t \otimes \I + (\id \otimes B_{+}) \circ \Delta(f_t)\]
and for a forest $f=g h$ with $g,h \in \HH$ multiplicatively by $\Delta(f) = \Delta(g)\Delta(h)$ and $\Delta(\I)= \I \otimes \I$. For example we have
\[\Delta(\begin{tikzpicture}[scale=0.3,baseline={([yshift=-.5ex]current bounding box.center)}]
\def\cz{5}
\def\wi{0.5}

\newcommand{\ci}[1]{	
	\fill[black] (#1) circle (\cz pt);
	\draw (#1) circle (\cz pt);
}

\coordinate (R) at (0,0);
\coordinate (r1) at (\wi,-1);
\coordinate (l1) at (-\wi,-1);

\draw (R) to (l1);
\draw (R) to (r1);

\ci{R};
\ci{l1};
\ci{r1};
\end{tikzpicture}) = \begin{tikzpicture}[scale=0.3,baseline={([yshift=-.5ex]current bounding box.center)}]
\def\cz{5}
\def\wi{0.5}

\newcommand{\ci}[1]{	
	\fill[black] (#1) circle (\cz pt);
	\draw (#1) circle (\cz pt);
}

\coordinate (R) at (0,0);
\coordinate (r1) at (\wi,-1);
\coordinate (l1) at (-\wi,-1);

\draw (R) to (l1);
\draw (R) to (r1);

\ci{R};
\ci{l1};
\ci{r1};
\end{tikzpicture} \otimes \I + \begin{tikzpicture}[scale=0.3,baseline={([yshift=-.5ex]current bounding box.center)}]
\def\cz{5}
\def\wi{0.5}

\newcommand{\ci}[1]{	
	\fill[black] (#1) circle (\cz pt);
	\draw (#1) circle (\cz pt);
}

\coordinate (R) at (0,0);
\ci{R}
\end{tikzpicture}\,\begin{tikzpicture}[scale=0.3,baseline={([yshift=-.5ex]current bounding box.center)}]
\def\cz{5}
\def\wi{0.5}

\newcommand{\ci}[1]{	
	\fill[black] (#1) circle (\cz pt);
	\draw (#1) circle (\cz pt);
}

\coordinate (R) at (0,0);
\ci{R}
\end{tikzpicture}  \otimes \begin{tikzpicture}[scale=0.3,baseline={([yshift=-.5ex]current bounding box.center)}]
\def\cz{5}
\def\wi{0.5}

\newcommand{\ci}[1]{	
	\fill[black] (#1) circle (\cz pt);
	\draw (#1) circle (\cz pt);
}

\coordinate (R) at (0,0);
\ci{R}
\end{tikzpicture} + 2 \,\begin{tikzpicture}[scale=0.3,baseline={([yshift=-.5ex]current bounding box.center)}]
\def\cz{5}
\def\wi{0.5}

\newcommand{\ci}[1]{	
	\fill[black] (#1) circle (\cz pt);
	\draw (#1) circle (\cz pt);
}

\coordinate (R) at (0,0);
\ci{R}
\end{tikzpicture}  \otimes \begin{tikzpicture}[scale=0.3,baseline={([yshift=-.5ex]current bounding box.center)}]
\def\cz{5}
\def\wi{0.5}

\newcommand{\ci}[1]{	
	\fill[black] (#1) circle (\cz pt);
	\draw (#1) circle (\cz pt);
}

\coordinate (R) at (0,0);
\coordinate (r1) at (0,-1);

\draw (R) to (r1);

\ci{R};
\ci{r1};

\end{tikzpicture}+ \I\otimes \begin{tikzpicture}[scale=0.3,baseline={([yshift=-.5ex]current bounding box.center)}]
\def\cz{5}
\def\wi{0.5}

\newcommand{\ci}[1]{	
	\fill[black] (#1) circle (\cz pt);
	\draw (#1) circle (\cz pt);
}

\coordinate (R) at (0,0);
\coordinate (r1) at (\wi,-1);
\coordinate (l1) at (-\wi,-1);

\draw (R) to (l1);
\draw (R) to (r1);

\ci{R};
\ci{l1};
\ci{r1};
\end{tikzpicture}\,.\]
 In \cite{T} the second named author uses the coproduct $\Delta$ to assign to a forest $f\in \HH$ a $\Q$-linear map on the space $\h = \Q\langle x,y \rangle$, called a rooted tree map, by the following:
\begin{definition}\label{def:rtm} For any non-empty forest $f \in \HH$, we define a $\Q$-linear map on $\h$, also denoted by $f$, recursively: For a word $w \in \h$ and a letter $u \in \{x,y\}$ we set
\begin{equation}\label{eq:defrtm}
f(w u):= M(\Delta(f)(w \otimes u))\,,
\end{equation}  
where $M(w_1 \otimes w_2) = w_1 w_2$ denotes the multiplication on $\h$. This reduces the calculation to $f(u)$ for a letter $u \in \{x,y\}$, which is defined by the following:
\begin{enumerate}[i)]
\item If $f=\begin{tikzpicture}[scale=0.3,baseline={([yshift=-.5ex]current bounding box.center)}]
\def\cz{5}
\def\wi{0.5}

\newcommand{\ci}[1]{	
	\fill[black] (#1) circle (\cz pt);
	\draw (#1) circle (\cz pt);
}

\coordinate (R) at (0,0);
\ci{R}
\end{tikzpicture}\,$, then $f(x) := xy$ and $f(y) := -xy$.
\item For a tree $t = B_+(f)$ we set $t(u) := R_y R_{x+2y}R_y^{-1} f(u)$,
where $R_v$ is the linear map given by $R_v(w)=wv$ ($v,w \in \h$).
\item If $f = gh$ is a forest with $g,h \neq \I$, then $f(u):=g(h(u))$.
\end{enumerate}
The rooted tree map of the empty forest $\I$ is given by the identity.
\end{definition}
By $\HM \subset \End(\h)$ we denote the space spanned by all rooted tree maps and by $\HM_d$ the space spanned by rooted tree maps of degree $d$ (number of vertices). Note that for any $f \in \HM_d$ with $d\geq 1$ we have $f(1)=0$, which follows by induction on the degree  of $f$ from \eqref{eq:defrtm} with $w=1$. 
In \cite{BT} it was shown that the derivation $\partial_n$ on $\h$, defined for $n\geq 1$ by $\partial_n(x) = x(x+y)^{n-1}y$ and $\partial_n(y) = -x(x+y)^{n-1}y$ can be written in terms of rooted tree maps. In particular the rooted tree map $\begin{tikzpicture}[scale=0.3,baseline={([yshift=-.5ex]current bounding box.center)}]
\def\cz{5}
\def\wi{0.5}

\newcommand{\ci}[1]{	
	\fill[black] (#1) circle (\cz pt);
	\draw (#1) circle (\cz pt);
}

\coordinate (R) at (0,0);
\ci{R}
\end{tikzpicture} = \partial_1$ is a derivation, which will be used for various calculations in the remaining parts of this work.

\section{The space $\FF$}\label{sec:F}
Denote by $\HH_d \subset \HH$ the subspace spanned by all rooted forests of degree $d$. 
In this section we will consider a subspace of $\HH_d$ defined recursively by $\FF_1 = \Q \cdot \begin{tikzpicture}[scale=0.3,baseline={([yshift=-.5ex]current bounding box.center)}]
\def\cz{5}
\def\wi{0.5}

\newcommand{\ci}[1]{	
	\fill[black] (#1) circle (\cz pt);
	\draw (#1) circle (\cz pt);
}

\coordinate (R) at (0,0);
\ci{R}
\end{tikzpicture}$ and for $d\geq 2$ by
\[   \FF_d = B_+\left(\FF_{d-1}\right) + \begin{tikzpicture}[scale=0.3,baseline={([yshift=-.5ex]current bounding box.center)}]
\def\cz{5}
\def\wi{0.5}

\newcommand{\ci}[1]{	
	\fill[black] (#1) circle (\cz pt);
	\draw (#1) circle (\cz pt);
}

\coordinate (R) at (0,0);
\ci{R}
\end{tikzpicture} \,\, \FF_{d-1}\,. \]
For example for $d=2,3,4$ the $\FF_d$ are given by
\begin{align*}
\FF_2 = \Q \,\, \begin{tikzpicture}[scale=0.3,baseline={([yshift=-.5ex]current bounding box.center)}]
\def\cz{5}
\def\wi{0.5}

\newcommand{\ci}[1]{	
	\fill[black] (#1) circle (\cz pt);
	\draw (#1) circle (\cz pt);
}

\coordinate (R) at (0,0);
\coordinate (r1) at (0,-1);

\draw (R) to (r1);

\ci{R};
\ci{r1};

\end{tikzpicture} + \Q\,\, \begin{tikzpicture}[scale=0.3,baseline={([yshift=-.5ex]current bounding box.center)}]
\def\cz{5}
\def\wi{0.5}

\newcommand{\ci}[1]{	
	\fill[black] (#1) circle (\cz pt);
	\draw (#1) circle (\cz pt);
}

\coordinate (R) at (0,0);
\ci{R}
\end{tikzpicture}\,\begin{tikzpicture}[scale=0.3,baseline={([yshift=-.5ex]current bounding box.center)}]
\def\cz{5}
\def\wi{0.5}

\newcommand{\ci}[1]{	
	\fill[black] (#1) circle (\cz pt);
	\draw (#1) circle (\cz pt);
}

\coordinate (R) at (0,0);
\ci{R}
\end{tikzpicture}\,,\qquad \FF_3 = \Q\,\,\begin{tikzpicture}[scale=0.3,baseline={([yshift=-.5ex]current bounding box.center)}]
\def\cz{5}
\def\wi{0.5}

\newcommand{\ci}[1]{	
	\fill[black] (#1) circle (\cz pt);
	\draw (#1) circle (\cz pt);
}

\coordinate (R) at (0,0);
\coordinate (r1) at (0,-1);
\coordinate (r2) at (0,-2);

\draw (R) to (r1);
\draw (r1) to (r2);
\ci{R};
\ci{r1};
\ci{r2};
\end{tikzpicture} +\Q\,\,\begin{tikzpicture}[scale=0.3,baseline={([yshift=-.5ex]current bounding box.center)}]
\def\cz{5}
\def\wi{0.5}

\newcommand{\ci}[1]{	
	\fill[black] (#1) circle (\cz pt);
	\draw (#1) circle (\cz pt);
}

\coordinate (R) at (0,0);
\coordinate (r1) at (\wi,-1);
\coordinate (l1) at (-\wi,-1);

\draw (R) to (l1);
\draw (R) to (r1);

\ci{R};
\ci{l1};
\ci{r1};
\end{tikzpicture}  +\Q\,\,\begin{tikzpicture}[scale=0.3,baseline={([yshift=-.5ex]current bounding box.center)}]
\def\cz{5}
\def\wi{0.5}

\newcommand{\ci}[1]{	
	\fill[black] (#1) circle (\cz pt);
	\draw (#1) circle (\cz pt);
}

\coordinate (R) at (0,0);
\coordinate (r1) at (0,-1);

\draw (R) to (r1);

\ci{R};
\ci{r1};

\end{tikzpicture} \, \begin{tikzpicture}[scale=0.3,baseline={([yshift=-.5ex]current bounding box.center)}]
\def\cz{5}
\def\wi{0.5}

\newcommand{\ci}[1]{	
	\fill[black] (#1) circle (\cz pt);
	\draw (#1) circle (\cz pt);
}

\coordinate (R) at (0,0);
\ci{R}
\end{tikzpicture} +\Q\,\,\begin{tikzpicture}[scale=0.3,baseline={([yshift=-.5ex]current bounding box.center)}]
\def\cz{5}
\def\wi{0.5}

\newcommand{\ci}[1]{	
	\fill[black] (#1) circle (\cz pt);
	\draw (#1) circle (\cz pt);
}

\coordinate (R) at (0,0);
\ci{R}
\end{tikzpicture} \, \begin{tikzpicture}[scale=0.3,baseline={([yshift=-.5ex]current bounding box.center)}]
\def\cz{5}
\def\wi{0.5}

\newcommand{\ci}[1]{	
	\fill[black] (#1) circle (\cz pt);
	\draw (#1) circle (\cz pt);
}

\coordinate (R) at (0,0);
\ci{R}
\end{tikzpicture} \, \begin{tikzpicture}[scale=0.3,baseline={([yshift=-.5ex]current bounding box.center)}]
\def\cz{5}
\def\wi{0.5}

\newcommand{\ci}[1]{	
	\fill[black] (#1) circle (\cz pt);
	\draw (#1) circle (\cz pt);
}

\coordinate (R) at (0,0);
\ci{R}
\end{tikzpicture} \,,\\
\FF_4 = \Q\,\,\begin{tikzpicture}[scale=0.3,baseline={([yshift=-.5ex]current bounding box.center)}]
\def\cz{5}
\def\wi{0.5}

\newcommand{\ci}[1]{	
	\fill[black] (#1) circle (\cz pt);
	\draw (#1) circle (\cz pt);
}

\coordinate (R) at (0,0);
\coordinate (r1) at (0,-1);
\coordinate (r2) at (0,-2);
\coordinate (r3) at (0,-3);

\draw (R) to (r1);
\draw (r1) to (r2);
\draw (r2) to (r3);
\ci{R};
\ci{r1};
\ci{r2};
\ci{r3};
\end{tikzpicture}+\Q\,\,\begin{tikzpicture}[scale=0.3,baseline={([yshift=-.5ex]current bounding box.center)}]
\def\cz{5}
\def\wi{0.5}

\newcommand{\ci}[1]{	
	\fill[black] (#1) circle (\cz pt);
	\draw (#1) circle (\cz pt);
}

\coordinate (R) at (0,0);

\coordinate (l1) at (0,-1);
\coordinate (l11) at (\wi,-2);
\coordinate (l12) at (-\wi,-2);

\draw (R) to (l1);

\draw (l1) to (l11);
\draw (l1) to (l12);

\ci{R}
\ci{l1}
\ci{l11}
\ci{l12}
\end{tikzpicture}+\Q\,\,\begin{tikzpicture}[scale=0.3,baseline={([yshift=-.5ex]current bounding box.center)}]
\def\cz{5}
\def\wi{0.5}

\newcommand{\ci}[1]{	
	\fill[black] (#1) circle (\cz pt);
	\draw (#1) circle (\cz pt);
}

\coordinate (R) at (0,0);
\coordinate (l1) at (-\wi,-1);
\coordinate (r1) at (\wi,-1);
\coordinate (l11) at (-\wi,-2);

\draw (R) to (l1);
\draw (R) to (r1);
\draw (l1) to (l11);

\ci{R}
\ci{l1}
\ci{r1}
\ci{l11}

\end{tikzpicture}+\Q\,\,\begin{tikzpicture}[scale=0.3,baseline={([yshift=-.5ex]current bounding box.center)}]
\def\cz{5}
\def\wi{0.5}

\newcommand{\ci}[1]{	
	\fill[black] (#1) circle (\cz pt);
	\draw (#1) circle (\cz pt);
}

\coordinate (R) at (0,0);
\coordinate (a1) at (-\wi,-1);
\coordinate (a2) at (0,-1);
\coordinate (a3) at (\wi,-1);

\draw (R) to (a1);
\draw (R) to (a2);
\draw (R) to (a3);

\ci{R}
\ci{a1}
\ci{a2}
\ci{a3}

\end{tikzpicture}+ \Q\,\,\begin{tikzpicture}[scale=0.3,baseline={([yshift=-.5ex]current bounding box.center)}]
\def\cz{5}
\def\wi{0.5}

\newcommand{\ci}[1]{	
	\fill[black] (#1) circle (\cz pt);
	\draw (#1) circle (\cz pt);
}

\coordinate (R) at (0,0);
\coordinate (r1) at (0,-1);
\coordinate (r2) at (0,-2);

\draw (R) to (r1);
\draw (r1) to (r2);
\ci{R};
\ci{r1};
\ci{r2};
\end{tikzpicture}\, \begin{tikzpicture}[scale=0.3,baseline={([yshift=-.5ex]current bounding box.center)}]
\def\cz{5}
\def\wi{0.5}

\newcommand{\ci}[1]{	
	\fill[black] (#1) circle (\cz pt);
	\draw (#1) circle (\cz pt);
}

\coordinate (R) at (0,0);
\ci{R}
\end{tikzpicture} +\Q\,\,\begin{tikzpicture}[scale=0.3,baseline={([yshift=-.5ex]current bounding box.center)}]
\def\cz{5}
\def\wi{0.5}

\newcommand{\ci}[1]{	
	\fill[black] (#1) circle (\cz pt);
	\draw (#1) circle (\cz pt);
}

\coordinate (R) at (0,0);
\coordinate (r1) at (\wi,-1);
\coordinate (l1) at (-\wi,-1);

\draw (R) to (l1);
\draw (R) to (r1);

\ci{R};
\ci{l1};
\ci{r1};
\end{tikzpicture}\, \begin{tikzpicture}[scale=0.3,baseline={([yshift=-.5ex]current bounding box.center)}]
\def\cz{5}
\def\wi{0.5}

\newcommand{\ci}[1]{	
	\fill[black] (#1) circle (\cz pt);
	\draw (#1) circle (\cz pt);
}

\coordinate (R) at (0,0);
\ci{R}
\end{tikzpicture}  +\Q\,\,\begin{tikzpicture}[scale=0.3,baseline={([yshift=-.5ex]current bounding box.center)}]
\def\cz{5}
\def\wi{0.5}

\newcommand{\ci}[1]{	
	\fill[black] (#1) circle (\cz pt);
	\draw (#1) circle (\cz pt);
}

\coordinate (R) at (0,0);
\coordinate (r1) at (0,-1);

\draw (R) to (r1);

\ci{R};
\ci{r1};

\end{tikzpicture} \, \begin{tikzpicture}[scale=0.3,baseline={([yshift=-.5ex]current bounding box.center)}]
\def\cz{5}
\def\wi{0.5}

\newcommand{\ci}[1]{	
	\fill[black] (#1) circle (\cz pt);
	\draw (#1) circle (\cz pt);
}

\coordinate (R) at (0,0);
\ci{R}
\end{tikzpicture}\, \begin{tikzpicture}[scale=0.3,baseline={([yshift=-.5ex]current bounding box.center)}]
\def\cz{5}
\def\wi{0.5}

\newcommand{\ci}[1]{	
	\fill[black] (#1) circle (\cz pt);
	\draw (#1) circle (\cz pt);
}

\coordinate (R) at (0,0);
\ci{R}
\end{tikzpicture} +\Q\,\,\begin{tikzpicture}[scale=0.3,baseline={([yshift=-.5ex]current bounding box.center)}]
\def\cz{5}
\def\wi{0.5}

\newcommand{\ci}[1]{	
	\fill[black] (#1) circle (\cz pt);
	\draw (#1) circle (\cz pt);
}

\coordinate (R) at (0,0);
\ci{R}
\end{tikzpicture} \, \begin{tikzpicture}[scale=0.3,baseline={([yshift=-.5ex]current bounding box.center)}]
\def\cz{5}
\def\wi{0.5}

\newcommand{\ci}[1]{	
	\fill[black] (#1) circle (\cz pt);
	\draw (#1) circle (\cz pt);
}

\coordinate (R) at (0,0);
\ci{R}
\end{tikzpicture} \, \begin{tikzpicture}[scale=0.3,baseline={([yshift=-.5ex]current bounding box.center)}]
\def\cz{5}
\def\wi{0.5}

\newcommand{\ci}[1]{	
	\fill[black] (#1) circle (\cz pt);
	\draw (#1) circle (\cz pt);
}

\coordinate (R) at (0,0);
\ci{R}
\end{tikzpicture}\, \begin{tikzpicture}[scale=0.3,baseline={([yshift=-.5ex]current bounding box.center)}]
\def\cz{5}
\def\wi{0.5}

\newcommand{\ci}[1]{	
	\fill[black] (#1) circle (\cz pt);
	\draw (#1) circle (\cz pt);
}

\coordinate (R) at (0,0);
\ci{R}
\end{tikzpicture}\,.
\end{align*}
Notice that the space $\FF=\bigoplus_{d\geq 1} \FF_d \subset \HH$ is a subspace of $\HH$, but not a subalgebra since for example $\begin{tikzpicture}[scale=0.3,baseline={([yshift=-.5ex]current bounding box.center)}]
\def\cz{5}
\def\wi{0.5}

\newcommand{\ci}[1]{	
	\fill[black] (#1) circle (\cz pt);
	\draw (#1) circle (\cz pt);
}

\coordinate (R) at (0,0);
\coordinate (r1) at (0,-1);

\draw (R) to (r1);

\ci{R};
\ci{r1};

\end{tikzpicture} \in \FF_2$ but $\begin{tikzpicture}[scale=0.3,baseline={([yshift=-.5ex]current bounding box.center)}]
\def\cz{5}
\def\wi{0.5}

\newcommand{\ci}[1]{	
	\fill[black] (#1) circle (\cz pt);
	\draw (#1) circle (\cz pt);
}

\coordinate (R) at (0,0);
\coordinate (r1) at (0,-1);

\draw (R) to (r1);

\ci{R};
\ci{r1};

\end{tikzpicture}\,\begin{tikzpicture}[scale=0.3,baseline={([yshift=-.5ex]current bounding box.center)}]
\def\cz{5}
\def\wi{0.5}

\newcommand{\ci}[1]{	
	\fill[black] (#1) circle (\cz pt);
	\draw (#1) circle (\cz pt);
}

\coordinate (R) at (0,0);
\coordinate (r1) at (0,-1);

\draw (R) to (r1);

\ci{R};
\ci{r1};

\end{tikzpicture} \, \notin \FF_4$.
By definition we have  $\dim_\Q \FF_d = 2^{d-1}$. Denote by $\FM_d \subset \End(\h)$ the space spanned by all rooted tree maps corresponding to the rooted trees in $\FF_d$ and set $\FM=\bigoplus_{d\geq 1} \FM_d \subset \HM$. The main goal of this section is to prove the following.
\begin{theorem}\label{thm:isom}We have an isomorphism of $\Q$-vector spaces 
\begin{align*}
\W: \FM &\longrightarrow x \h y \\
f & \longmapsto f(x) \,.
\end{align*}
\end{theorem}
Before we can prove Theorem \ref{thm:isom} at the end of this section,  we need to introduce some notation and prove some Lemma. Define the $2^{n-1}\times2^n$-matrix $A_n$ by $A_1 = (1 \,\, 1)$ and for $n\geq 2$ recursively by
\[A_n := \left(\begin{array}{c;{2pt/2pt}c;{2pt/2pt}c;{2pt/2pt}c} 
\multicolumn{2}{c;{2pt/2pt}}{A_{n-1}} & E_{2^{n-2}} & 0 \\ \hdashline[2pt/2pt]
0 & E_{2^{n-2}}  & \multicolumn{2}{c}{A_{n-1}} 
\end{array}
\right) =: \left( A^{(1)}_n \mid A^{(2)}_n \right)\,, \]
where $E_n$ denotes the $n \times n$ identity matrix. The matrices $A^{(j)}_n$ are both $2^{n-1} \times 2^{n-1}$ square matrices. Denote for $d\geq 2$ by $\w_d \in (\h^0_d)^{2^{d-2}}$ the vector of all $2^{d-2}$ monomials in $\h_0$ of degree $d$, ordered in lexicographical order ($x<y$) from the top to the bottom. For example we have for $d=2,3,4$
\[\w_2 = \begin{pmatrix}  xy \end{pmatrix} \,,\qquad \w_3 = \begin{pmatrix}  x^2 y \\ x y^2 \end{pmatrix}\,,\qquad \w_4 = \begin{pmatrix}  x^3 y \\ x^2 y^2 \\ xyxy \\ x y^3\end{pmatrix}\,.  \]
For a rooted forest $f\in \HH$ we denote by $f(\w_d) \in (\h^0)^{2^{d-2}}$ the component-wise evaluation of the rooted tree map $f$ on the entries of $\w_d$.

In the following we denote for $v,w \in \h$ by $R_v$ and $L_v$ the linear maps given by $R_v(w)=wv$ and $L_v(w) = vw$.
\begin{lemma}\label{lem:lem1}
For all $d\geq 2$ we have $\begin{tikzpicture}[scale=0.3,baseline={([yshift=-.5ex]current bounding box.center)}]
\def\cz{5}
\def\wi{0.5}

\newcommand{\ci}[1]{	
	\fill[black] (#1) circle (\cz pt);
	\draw (#1) circle (\cz pt);
}

\coordinate (R) at (0,0);
\ci{R}
\end{tikzpicture}(\w_d) \equiv A_{d-1} \w_{d+1} \mod{2}$.
\end{lemma}
\begin{proof} We prove this statement by induction on $d$. For $d=2$ this follows immediately, since $\begin{tikzpicture}[scale=0.3,baseline={([yshift=-.5ex]current bounding box.center)}]
\def\cz{5}
\def\wi{0.5}

\newcommand{\ci}[1]{	
	\fill[black] (#1) circle (\cz pt);
	\draw (#1) circle (\cz pt);
}

\coordinate (R) at (0,0);
\ci{R}
\end{tikzpicture}(\w_2)=(\begin{tikzpicture}[scale=0.3,baseline={([yshift=-.5ex]current bounding box.center)}]
\def\cz{5}
\def\wi{0.5}

\newcommand{\ci}[1]{	
	\fill[black] (#1) circle (\cz pt);
	\draw (#1) circle (\cz pt);
}

\coordinate (R) at (0,0);
\ci{R}
\end{tikzpicture}(xy))=(xyy-xxy)$ and $A_1 w_3 = (xxy + xyy)$.
By definition of $\w_d$ one can check that 
\begin{equation}\label{eq:wd}
\w_d = \begin{pmatrix}
L_x(\w_{d-1})\\
L_{xy} L_x^{-1}(\w_{d-1})
\end{pmatrix} = \begin{pmatrix}
L_{x^2}(x \w'_{d-2}) \\
L_{x^2}(y \w'_{d-2}) \\
L_{xy}(x \w'_{d-2}) \\
L_{xy}(y \w'_{d-2}) 
\end{pmatrix}\,,
\end{equation}
where we write $\w'_{d-2} := L_{x}^{-1}(\w_{d-2})$.
Using that $\begin{tikzpicture}[scale=0.3,baseline={([yshift=-.5ex]current bounding box.center)}]
\def\cz{5}
\def\wi{0.5}

\newcommand{\ci}[1]{	
	\fill[black] (#1) circle (\cz pt);
	\draw (#1) circle (\cz pt);
}

\coordinate (R) at (0,0);
\ci{R}
\end{tikzpicture}$ is a derivation with $\begin{tikzpicture}[scale=0.3,baseline={([yshift=-.5ex]current bounding box.center)}]
\def\cz{5}
\def\wi{0.5}

\newcommand{\ci}[1]{	
	\fill[black] (#1) circle (\cz pt);
	\draw (#1) circle (\cz pt);
}

\coordinate (R) at (0,0);
\ci{R}
\end{tikzpicture}(x)=xy$ and $\begin{tikzpicture}[scale=0.3,baseline={([yshift=-.5ex]current bounding box.center)}]
\def\cz{5}
\def\wi{0.5}

\newcommand{\ci}[1]{	
	\fill[black] (#1) circle (\cz pt);
	\draw (#1) circle (\cz pt);
}

\coordinate (R) at (0,0);
\ci{R}
\end{tikzpicture}(y)=-xy$ yields
\begin{align}\label{eq:lem1_1}\begin{split}
\begin{tikzpicture}[scale=0.3,baseline={([yshift=-.5ex]current bounding box.center)}]
\def\cz{5}
\def\wi{0.5}

\newcommand{\ci}[1]{	
	\fill[black] (#1) circle (\cz pt);
	\draw (#1) circle (\cz pt);
}

\coordinate (R) at (0,0);
\ci{R}
\end{tikzpicture}(L_x(\w_{d-1})) &= \begin{tikzpicture}[scale=0.3,baseline={([yshift=-.5ex]current bounding box.center)}]
\def\cz{5}
\def\wi{0.5}

\newcommand{\ci}[1]{	
	\fill[black] (#1) circle (\cz pt);
	\draw (#1) circle (\cz pt);
}

\coordinate (R) at (0,0);
\ci{R}
\end{tikzpicture}(x^2 w'_{d-1}) = \begin{tikzpicture}[scale=0.3,baseline={([yshift=-.5ex]current bounding box.center)}]
\def\cz{5}
\def\wi{0.5}

\newcommand{\ci}[1]{	
	\fill[black] (#1) circle (\cz pt);
	\draw (#1) circle (\cz pt);
}

\coordinate (R) at (0,0);
\ci{R}
\end{tikzpicture}(x) x w'_{d-1} + x \,\begin{tikzpicture}[scale=0.3,baseline={([yshift=-.5ex]current bounding box.center)}]
\def\cz{5}
\def\wi{0.5}

\newcommand{\ci}[1]{	
	\fill[black] (#1) circle (\cz pt);
	\draw (#1) circle (\cz pt);
}

\coordinate (R) at (0,0);
\ci{R}
\end{tikzpicture}(x) \w'_{d-1} + x^2 \begin{tikzpicture}[scale=0.3,baseline={([yshift=-.5ex]current bounding box.center)}]
\def\cz{5}
\def\wi{0.5}

\newcommand{\ci}[1]{	
	\fill[black] (#1) circle (\cz pt);
	\draw (#1) circle (\cz pt);
}

\coordinate (R) at (0,0);
\ci{R}
\end{tikzpicture}(w'_{d-1})\\
&=L_{xy}(w_{d-1}) + L_{x^2}(L_y + \begin{tikzpicture}[scale=0.3,baseline={([yshift=-.5ex]current bounding box.center)}]
\def\cz{5}
\def\wi{0.5}

\newcommand{\ci}[1]{	
	\fill[black] (#1) circle (\cz pt);
	\draw (#1) circle (\cz pt);
}

\coordinate (R) at (0,0);
\ci{R}
\end{tikzpicture})(\w'_{d-1})
\end{split}
\end{align}
and 
\begin{align}\label{eq:lem1_2} \begin{split}
\begin{tikzpicture}[scale=0.3,baseline={([yshift=-.5ex]current bounding box.center)}]
\def\cz{5}
\def\wi{0.5}

\newcommand{\ci}[1]{	
	\fill[black] (#1) circle (\cz pt);
	\draw (#1) circle (\cz pt);
}

\coordinate (R) at (0,0);
\ci{R}
\end{tikzpicture}(L_{xy} L_x^{-1}(\w_{d-1})) &= \begin{tikzpicture}[scale=0.3,baseline={([yshift=-.5ex]current bounding box.center)}]
\def\cz{5}
\def\wi{0.5}

\newcommand{\ci}[1]{	
	\fill[black] (#1) circle (\cz pt);
	\draw (#1) circle (\cz pt);
}

\coordinate (R) at (0,0);
\ci{R}
\end{tikzpicture}(xy \w'_{d-1}) = \begin{tikzpicture}[scale=0.3,baseline={([yshift=-.5ex]current bounding box.center)}]
\def\cz{5}
\def\wi{0.5}

\newcommand{\ci}[1]{	
	\fill[black] (#1) circle (\cz pt);
	\draw (#1) circle (\cz pt);
}

\coordinate (R) at (0,0);
\ci{R}
\end{tikzpicture}(x) y \w'_{d-1} + x\, \begin{tikzpicture}[scale=0.3,baseline={([yshift=-.5ex]current bounding box.center)}]
\def\cz{5}
\def\wi{0.5}

\newcommand{\ci}[1]{	
	\fill[black] (#1) circle (\cz pt);
	\draw (#1) circle (\cz pt);
}

\coordinate (R) at (0,0);
\ci{R}
\end{tikzpicture}(y) \w'_{d-1}+xy \,\begin{tikzpicture}[scale=0.3,baseline={([yshift=-.5ex]current bounding box.center)}]
\def\cz{5}
\def\wi{0.5}

\newcommand{\ci}[1]{	
	\fill[black] (#1) circle (\cz pt);
	\draw (#1) circle (\cz pt);
}

\coordinate (R) at (0,0);
\ci{R}
\end{tikzpicture}(\w'_{d-1})\\
&\equiv L_{xy}(L_y + \begin{tikzpicture}[scale=0.3,baseline={([yshift=-.5ex]current bounding box.center)}]
\def\cz{5}
\def\wi{0.5}

\newcommand{\ci}[1]{	
	\fill[black] (#1) circle (\cz pt);
	\draw (#1) circle (\cz pt);
}

\coordinate (R) at (0,0);
\ci{R}
\end{tikzpicture})(\w'_{d-1}) + L_{x^2}(y \w'_{d-1})  \mod{2}\,.
\end{split}
\end{align}
Now by assumption we have
\[A_{d-2} \w_d \equiv \begin{tikzpicture}[scale=0.3,baseline={([yshift=-.5ex]current bounding box.center)}]
\def\cz{5}
\def\wi{0.5}

\newcommand{\ci}[1]{	
	\fill[black] (#1) circle (\cz pt);
	\draw (#1) circle (\cz pt);
}

\coordinate (R) at (0,0);
\ci{R}
\end{tikzpicture}(\w_{d-1}) = \begin{tikzpicture}[scale=0.3,baseline={([yshift=-.5ex]current bounding box.center)}]
\def\cz{5}
\def\wi{0.5}

\newcommand{\ci}[1]{	
	\fill[black] (#1) circle (\cz pt);
	\draw (#1) circle (\cz pt);
}

\coordinate (R) at (0,0);
\ci{R}
\end{tikzpicture}(x \w'_{d-1}) = L_{xy}(\w'_{d-1}) + L_x \,\begin{tikzpicture}[scale=0.3,baseline={([yshift=-.5ex]current bounding box.center)}]
\def\cz{5}
\def\wi{0.5}

\newcommand{\ci}[1]{	
	\fill[black] (#1) circle (\cz pt);
	\draw (#1) circle (\cz pt);
}

\coordinate (R) at (0,0);
\ci{R}
\end{tikzpicture}(\w'_{d-1})  \mod{2}\,. \]
This together with \eqref{eq:lem1_1} and \eqref{eq:lem1_2} yields
\begin{align*}
\begin{tikzpicture}[scale=0.3,baseline={([yshift=-.5ex]current bounding box.center)}]
\def\cz{5}
\def\wi{0.5}

\newcommand{\ci}[1]{	
	\fill[black] (#1) circle (\cz pt);
	\draw (#1) circle (\cz pt);
}

\coordinate (R) at (0,0);
\ci{R}
\end{tikzpicture}(L_x(\w_{d-1})) &\equiv A_{d-2}(L_x(\w_d)) + L_{xy}(\w_{d-1}) \mod{2} \,,\\
\begin{tikzpicture}[scale=0.3,baseline={([yshift=-.5ex]current bounding box.center)}]
\def\cz{5}
\def\wi{0.5}

\newcommand{\ci}[1]{	
	\fill[black] (#1) circle (\cz pt);
	\draw (#1) circle (\cz pt);
}

\coordinate (R) at (0,0);
\ci{R}
\end{tikzpicture}(L_{xy}L_x^{-1}(\w_{d-1})) &\equiv A_{d-2}(L_{xy} L_x^{-1}(\w_d))+L_{x^2} L_y(\w'_{d-1}) \mod{2}\,.
\end{align*}
This concludes the statement in the Lemma, since by definition of $A_{d-1}$ and \eqref{eq:wd} the right-hand side gives the entries of $A_{d-1}\w_{d+1}$ .
\end{proof}
Define for $d\geq 1$ the $2^{d} \times 2^{d-1}$-matrices
\[\E^{(1)}_{d} = {\tiny\begin{pmatrix}
1 & 0  & \dots & 0\\[-6pt]
0 & 0  &  & \vdots \\[-6pt]
0 & 1  &  & \vdots \\[-6pt]
0 & 0  &  & \vdots \\[-6pt]
\vdots& &\ddots & \vdots \\[-6pt]
\vdots& & & 1\\[-3pt]
0 &\dots  &\dots & 0
\end{pmatrix}}  \,,\qquad \E^{(2)}_{d} =  {\tiny\begin{pmatrix}
0 & 0  & \dots & 0\\[-6pt]
1 & 0  &  & \vdots \\[-6pt]
0 & 0  &  & \vdots \\[-6pt]
0 & 1  &  & \vdots \\[-6pt]
\vdots& &\ddots & \vdots \\[-6pt]
\vdots& & & 0\\[-3pt]
0 &\dots  &\dots & 1
\end{pmatrix}}\,. \]

\begin{lemma}\label{lem:lem3}
We have for all $d\geq 2$
\begin{equation}
\w_{d+1} = \E^{(2)}_{d-1} R_y \w_d + \E^{(1)}_{d-1} R_{xy} R_{y}^{-1} \w_d \,.
\end{equation}
\end{lemma}
\begin{proof}
This follows from the fact that if one orders the monomials $m^{(d)}_\ast$ in $x\h y$ of degree $d$ lexicographically by $m^{(d)}_1 < m^{(d)}_2 < \dots < m^{(d)}_{2^{d-2}}$, then $m^{(d+1)}_{2j-1} = R_{xy} R_{y^{-1}} m^{(d)}_j$ and $m^{(d+1)}_{2j}= R_y m^{(d)}_j$ for $j=1,\dots,2^{d-2}$.
\end{proof}

Define on $\h$ the anti-automorphism $\tau$ by $\tau(x)=y$ and $\tau(y)=x$. Clearly we have $\tau(x\h y) = x \h y$ which we can use to define for $d\geq 1$ the unique permutation matrix $T_d \in M_{2^{d-2}}(\Z)$ satisfying
\[ T_d \w_d = \tau(\w_d) \,. \] 

\begin{lemma}\label{lem:lem4}
We have for all $d\geq 2$
\begin{equation}\label{eq:lem4}
\det\left( A_{d-1} \E_{d-1}^{(2)} \right) \equiv 1 \mod{2}\,.
\end{equation}
\end{lemma}
\begin{proof}
By Lemma \ref{lem:lem1} and \ref{lem:lem3} we have
\begin{align}\label{eq:lem4eq1}\begin{split}
\begin{tikzpicture}[scale=0.3,baseline={([yshift=-.5ex]current bounding box.center)}]
\def\cz{5}
\def\wi{0.5}

\newcommand{\ci}[1]{	
	\fill[black] (#1) circle (\cz pt);
	\draw (#1) circle (\cz pt);
}

\coordinate (R) at (0,0);
\ci{R}
\end{tikzpicture}(\w_d) &\equiv A_{d-1} \w_{d+1} \mod{2}\\
&= A_{d-1}\E^{(2)}_{d-1} R_y (\w_d) +  A_{d-1}\E^{(1)}_{d-1} R_{xy} R_{y}^{-1} (\w_d) \,.
\end{split}
\end{align}
Since for all $w \in \h$ we have $\begin{tikzpicture}[scale=0.3,baseline={([yshift=-.5ex]current bounding box.center)}]
\def\cz{5}
\def\wi{0.5}

\newcommand{\ci}[1]{	
	\fill[black] (#1) circle (\cz pt);
	\draw (#1) circle (\cz pt);
}

\coordinate (R) at (0,0);
\ci{R}
\end{tikzpicture}(\tau(wy)) = xy \tau(w) + x\,\begin{tikzpicture}[scale=0.3,baseline={([yshift=-.5ex]current bounding box.center)}]
\def\cz{5}
\def\wi{0.5}

\newcommand{\ci}[1]{	
	\fill[black] (#1) circle (\cz pt);
	\draw (#1) circle (\cz pt);
}

\coordinate (R) at (0,0);
\ci{R}
\end{tikzpicture}(\tau(w))$ and $\tau(\begin{tikzpicture}[scale=0.3,baseline={([yshift=-.5ex]current bounding box.center)}]
\def\cz{5}
\def\wi{0.5}

\newcommand{\ci}[1]{	
	\fill[black] (#1) circle (\cz pt);
	\draw (#1) circle (\cz pt);
}

\coordinate (R) at (0,0);
\ci{R}
\end{tikzpicture}(wy))= x \tau(\begin{tikzpicture}[scale=0.3,baseline={([yshift=-.5ex]current bounding box.center)}]
\def\cz{5}
\def\wi{0.5}

\newcommand{\ci}[1]{	
	\fill[black] (#1) circle (\cz pt);
	\draw (#1) circle (\cz pt);
}

\coordinate (R) at (0,0);
\ci{R}
\end{tikzpicture}(w))-xy(\tau(w)) $ the operators $\begin{tikzpicture}[scale=0.3,baseline={([yshift=-.5ex]current bounding box.center)}]
\def\cz{5}
\def\wi{0.5}

\newcommand{\ci}[1]{	
	\fill[black] (#1) circle (\cz pt);
	\draw (#1) circle (\cz pt);
}

\coordinate (R) at (0,0);
\ci{R}
\end{tikzpicture}$ and $\tau$ commute modulo $2$. Together with the definition of the matrix $T_d$, Lemma \ref{lem:lem1}, $\tau L_x = R_y \tau$ and $\tau L_{xy} = R_{xy}\tau$ we get
\begin{align*}
T_d \,\begin{tikzpicture}[scale=0.3,baseline={([yshift=-.5ex]current bounding box.center)}]
\def\cz{5}
\def\wi{0.5}

\newcommand{\ci}[1]{	
	\fill[black] (#1) circle (\cz pt);
	\draw (#1) circle (\cz pt);
}

\coordinate (R) at (0,0);
\ci{R}
\end{tikzpicture}(\w_d) &= \begin{tikzpicture}[scale=0.3,baseline={([yshift=-.5ex]current bounding box.center)}]
\def\cz{5}
\def\wi{0.5}

\newcommand{\ci}[1]{	
	\fill[black] (#1) circle (\cz pt);
	\draw (#1) circle (\cz pt);
}

\coordinate (R) at (0,0);
\ci{R}
\end{tikzpicture}(\tau(\w_d)) \equiv \tau(\begin{tikzpicture}[scale=0.3,baseline={([yshift=-.5ex]current bounding box.center)}]
\def\cz{5}
\def\wi{0.5}

\newcommand{\ci}[1]{	
	\fill[black] (#1) circle (\cz pt);
	\draw (#1) circle (\cz pt);
}

\coordinate (R) at (0,0);
\ci{R}
\end{tikzpicture}(\w_d)) \equiv \tau( A_{d-1} w_{d+1}) \\
&=\tau\left(A^{(1)}_{d-1}L_x (\w_d) + A^{(2)}_{d-1} L_{xy} L^{-1}_{x}(\w_d)\right)\\
&= A^{(1)}_{d-1} T_d R_y(\w_d) + A^{(2)}_{d-1} T_d R_{xy} R_y^{-1}(\w_d) \mod{2}\,.
\end{align*}
Since the entries in the first summand of this expression end all in $y^2$ and the entries of the second expression end all in $xy$, we get by \eqref{eq:lem4eq1} the identities  $A_{d-1} \E_{d-1}^{(2)} \equiv T_d A^{(1)}_{d-1} T_d \mod{2}$ and $A_{d-1} \E_{d-1}^{(1)} \equiv T_d A^{(2)}_{d-1} T_d \mod{2}$. Since $A^{(1)}_{d-1}$ is an upper triangle matrix with $1$ on the diagonal we have $\det(A^{(1)}_{d-1})=1$ which together with  $\det(T_d)^2 = 1$ gives \eqref{eq:lem4}.
\end{proof}


Now define for $d\geq 1$ the matrices $B_d \in M_{2^{d-1}}(\Z)$ by $B_1 =(1)$ and for $d \geq 2$ by
\[ B_d = \left(
\begin{array}{c} 
B_{d-1} (\E_{d-1}^{(1)})^t\\ \hdashline[2pt/2pt]
B_{d-1} A_{d-1}
\end{array}\,
\right) \,.\]
Notice that $B_{d-1} (\E_{d-1}^{(1)})^t$ is just $B_{d-1}$ with zero columns added at the even places, i.e. if $B_{d-1} = (b_1, \dots,b_{2^{d-2}})$ then $B_{d-1} (\E_{d-1}^{(2)})^t = (b_1,0,b_2,0,\dots,b_{2^{d-2}},0)$. 

\begin{lemma}\label{lem:lem2}
For all $d\geq 1$ we have $\det\left( B_d \right) \equiv 1 \mod{2}$.
\end{lemma}
\begin{proof}
With the matrix $\E_d := (\E^{(1)}_d \mid \E^{(2)}_d) \in M_{2^{d-1}}(\Z)$  one checks that 
\[ B_d  \E_d = \left(\begin{array}{c;{2pt/2pt}c} 
B_{d-1} & 0 \\ \hdashline[2pt/2pt] 
B_{d-1} A_{d-1} \E^{(1)}_{d-1} & B_{d-1} A_{d-1} \E^{(2)}_{d-1}
\end{array}
\right) \,.\]
The result now follows inductively together with Lemma \ref{lem:lem4}.
\end{proof}

Before we can finally prove Theorem \ref{thm:isom} we will now give the connection of the map $\W: \FF \rightarrow x \h y$ given by $f \mapsto f(x)$ and the matrix $B_d$. For this we define the vector $\f_d \in \FF_d^{2^{d-1}}$ by $f_1 = (\begin{tikzpicture}[scale=0.3,baseline={([yshift=-.5ex]current bounding box.center)}]
\def\cz{5}
\def\wi{0.5}

\newcommand{\ci}[1]{	
	\fill[black] (#1) circle (\cz pt);
	\draw (#1) circle (\cz pt);
}

\coordinate (R) at (0,0);
\ci{R}
\end{tikzpicture})$ and 
\[  \f_d = \left(
\begin{array}{c} 
B_{+}(\f_{d-1})\\\hdashline[2pt/2pt] 
\begin{tikzpicture}[scale=0.3,baseline={([yshift=-.5ex]current bounding box.center)}]
\def\cz{5}
\def\wi{0.5}

\newcommand{\ci}[1]{	
	\fill[black] (#1) circle (\cz pt);
	\draw (#1) circle (\cz pt);
}

\coordinate (R) at (0,0);
\ci{R}
\end{tikzpicture} \, \f_{d-1}
\end{array}\,
\right) \,.\]
So in particular $\f_d$ contains all generators of $\FF_d$. By $\f_d(x)\in (x \h y)^{2^{d-1}}$ we will denote the vectors obtained by evaluating all entries at $x$.

\begin{lemma}\label{lem:prop1}
For all $d\geq 1$ we have $f_d(x) \equiv B_d \w_{d+1} \mod{2}$.
\end{lemma}
\begin{proof}
We will again use induction on $d$. For $d=1$ the statement can be checked easily. Now for $d\geq 2$ we get by assumption, the definition of rooted tree maps and Lemma \ref{lem:lem1}
\begin{align*}
\f_d(x) &= \left(
\begin{array}{c} 
B_{+}(\f_{d-1})(x)\\\hdashline[2pt/2pt] 
\begin{tikzpicture}[scale=0.3,baseline={([yshift=-.5ex]current bounding box.center)}]
\def\cz{5}
\def\wi{0.5}

\newcommand{\ci}[1]{	
	\fill[black] (#1) circle (\cz pt);
	\draw (#1) circle (\cz pt);
}

\coordinate (R) at (0,0);
\ci{R}
\end{tikzpicture} \, \f_{d-1}(x)
\end{array}\,
\right) \equiv \left(\begin{array}{c} 
R_y R_x R_{y}^{-1}\f_{d-1}(x)\\\hdashline[2pt/2pt] 
\begin{tikzpicture}[scale=0.3,baseline={([yshift=-.5ex]current bounding box.center)}]
\def\cz{5}
\def\wi{0.5}

\newcommand{\ci}[1]{	
	\fill[black] (#1) circle (\cz pt);
	\draw (#1) circle (\cz pt);
}

\coordinate (R) at (0,0);
\ci{R}
\end{tikzpicture} \, \f_{d-1}(x)
\end{array}\,
\right) \\
&\equiv \left(\begin{array}{c} 
B_{d-1} R_y R_x R_{y}^{-1} \w_d \\\hdashline[2pt/2pt] 
B_{d-1}  \,\begin{tikzpicture}[scale=0.3,baseline={([yshift=-.5ex]current bounding box.center)}]
\def\cz{5}
\def\wi{0.5}

\newcommand{\ci}[1]{	
	\fill[black] (#1) circle (\cz pt);
	\draw (#1) circle (\cz pt);
}

\coordinate (R) at (0,0);
\ci{R}
\end{tikzpicture}(\w_d)
\end{array}\,
\right)  \equiv \left(\begin{array}{c} 
B_{d-1} R_y R_x R_{y}^{-1}(\w_d) \\\hdashline[2pt/2pt] 
B_{d-1}  A_{d-1} \w_d
\end{array}\,
\right) \mod{2} \,.
\end{align*}
Together with $B_{d-1} R_y R_x R_{y}^{-1}(\w_d) = B_{d-1} (\E_{d-1}^{(1)})^t \w_{d+1}$, which follows from the lexicographical ordering of the $\w_d$, we obtain the desired result. 
\end{proof}

\begin{proof}[Proof of Theorem \ref{thm:isom}] 
That $\W$ is an isomorphism follows now directly from Lemma \ref{lem:lem1} and \ref{lem:prop1}, which imply that every monomial in $x\h y$ of degree $d+1$ can be written as some $f(x)$ with $f\in \FF_d$. Therefore the map $\W$ is surjective on the degree graded parts $\FF_d$ and $(x\h y)_{d+1}$, which both have dimension $2^{d-1}$.
\end{proof}
\section{Kawashima relations}
In this section we want to give the proof of Theorem \ref{thm:main10}, i.e. that the rooted tree map relations imply the linear part of the Kawashima relation. Recall that with the automorphism $\varphi \in \Aut(\h)$ defined by  $\varphi(x) =  x+y$ and $\varphi(y) = -y$ the linear part of the Kawashima relations, i.e. the $m=1$ case in Theorem \ref{thm:kawashimarel}, states that
\begin{equation}
\varphi(v \ast w) \circledast y = L_x \varphi(v \ast w) \in \ker Z
\end{equation}
for any $w,v \in \h y$. In \cite{T} it was shown, that the rooted tree maps relations follow from the linear part of the Kawashima relations (see proof of \cite[Corollary 4.7]{T}), i.e. we have in particular
\[ \{ f(w) \mid f\in \FM,\, w \in \h^0 \} \subset  L_x \varphi(\h y \ast \h y) \,.  \]
We will now show that also the opposite inclusion holds.
\begin{theorem}\label{thm:main1}
We have
\[L_x \varphi(\h y \ast \h y)\subset \{ f(w) \mid f\in \FM,\, w \in \h^0 \} \,.\]
\end{theorem}
\begin{proof}
Due to Lemma 4.2 in \cite{T} we know that \[(1-\tau)(x\h y) \subset L_x \varphi(\h y \ast \h y) \,.\]
So we have $(1-\tau)(x\h y)= \tau (1-\tau)(x\h y) \subset \tau L_x \varphi(\h y \ast \h y) $ and hence
\begin{align*}
 L_x \varphi(\h y \ast \h y) &= ((1-\tau) + \tau)  L_x \varphi(\h y \ast \h y) \subset (1-\tau)(x\h y) + \tau L_x \varphi(\h y \ast \h y) \\
 &\subset \tau L_x \varphi(\h y \ast \h y) \,.
\end{align*}
The opposite inclusion $\tau L_x \varphi(\h y \ast \h y) \subset  L_x \varphi(\h y \ast \h y)$ was shown in the proof of Corollary 4.7 in \cite{T} and therefore we have 
\begin{equation}\label{eq:keq1}
 L_x \varphi(\h y \ast \h y) = \tau L_x \varphi(\h y \ast \h y)\,.
\end{equation}
Following \cite{T} we define $\chi_x := \tau L_x \varphi$. Due to Corollary 4.5. in \cite{T} there exist for any $ f\in \FM$ a $w \in \h y$ such that
\begin{equation}\label{eq:fhident}
f \chi_x = \chi_x H_w \,,
\end{equation} 
where $H_w(v) = w \ast v$ for $w,v \in \h y$. The element $w$ is uniquely determined by $w=\chi_x^{-1} f(y)$. Since we have for any rooted tree map $f \in \FM$ that $f(y) = - f(x)$, we get $-\chi_x(w) = f(x)$. In Theorem \ref{thm:isom} we proved that the map $\W: \FF \rightarrow x \h y$ with $f  \mapsto f(x)$ is an isomorphism and hence for an arbitrary  $w\in \h y$ the rooted tree map $f = -\W^{-1}(\chi_x(w)) \in \FF$ satisfies  \eqref{eq:fhident}.

Now let $w,v$ be arbitrary elements in $\h y$. To proof the statement of the theorem we want to show that we can find a rooted tree map  $f\in \FM$ and a $u \in \h^0$, such that $L_x\varphi(w \ast v) = f(u)$. Because of \eqref{eq:keq1} it suffices to show that we can find such an $f$ and $u$ with $\tau L_x\varphi(w \ast v) = \chi_x H_w(v)= f(u) = f \chi_x ( \chi_x^{-1}u )$. By the above discussion we can choose $u= \chi_x(v)$ and $f=-\W^{-1}(\chi_x(w))$.
\end{proof}

\begin{remark}
The space $\FF$ is a proper subset of $\HH$, since $\begin{tikzpicture}[scale=0.3,baseline={([yshift=-.5ex]current bounding box.center)}]
\def\cz{5}
\def\wi{0.5}

\newcommand{\ci}[1]{	
	\fill[black] (#1) circle (\cz pt);
	\draw (#1) circle (\cz pt);
}

\coordinate (R) at (0,0);
\coordinate (r1) at (0,-1);

\draw (R) to (r1);

\ci{R};
\ci{r1};

\end{tikzpicture}\,\begin{tikzpicture}[scale=0.3,baseline={([yshift=-.5ex]current bounding box.center)}]
\def\cz{5}
\def\wi{0.5}

\newcommand{\ci}[1]{	
	\fill[black] (#1) circle (\cz pt);
	\draw (#1) circle (\cz pt);
}

\coordinate (R) at (0,0);
\coordinate (r1) at (0,-1);

\draw (R) to (r1);

\ci{R};
\ci{r1};

\end{tikzpicture} \in \HH$ but $\begin{tikzpicture}[scale=0.3,baseline={([yshift=-.5ex]current bounding box.center)}]
\def\cz{5}
\def\wi{0.5}

\newcommand{\ci}[1]{	
	\fill[black] (#1) circle (\cz pt);
	\draw (#1) circle (\cz pt);
}

\coordinate (R) at (0,0);
\coordinate (r1) at (0,-1);

\draw (R) to (r1);

\ci{R};
\ci{r1};

\end{tikzpicture}\,\begin{tikzpicture}[scale=0.3,baseline={([yshift=-.5ex]current bounding box.center)}]
\def\cz{5}
\def\wi{0.5}

\newcommand{\ci}[1]{	
	\fill[black] (#1) circle (\cz pt);
	\draw (#1) circle (\cz pt);
}

\coordinate (R) at (0,0);
\coordinate (r1) at (0,-1);

\draw (R) to (r1);

\ci{R};
\ci{r1};

\end{tikzpicture} \notin \FF$. Due to Theorem \ref{thm:main1} all rooted tree maps relations are already obtained by considering just the rooted trees in $\FF$, that is the space $\FM$. Due to numerical experiments we actually expect that the space $\FM$ is  the whole space $\HM$, i.e. the recursively defined elements at the beginning of Section \ref{sec:F} might give a basis of the space of rooted tree maps and in particular $\dim_\Q \HM_d = 2^{d-1}$. Since the number of rooted tree maps in degree $d$ is larger than $2^{d-1}$ for $d\geq 4$, this would give linear relations between rooted tree maps. For example we expect that we have for all $w\in \h$ 
\[ \begin{tikzpicture}[scale=0.3,baseline={([yshift=-.5ex]current bounding box.center)}]
\def\cz{5}
\def\wi{0.5}

\newcommand{\ci}[1]{	
	\fill[black] (#1) circle (\cz pt);
	\draw (#1) circle (\cz pt);
}

\coordinate (R) at (0,0);
\coordinate (r1) at (0,-1);

\draw (R) to (r1);

\ci{R};
\ci{r1};

\end{tikzpicture}\,\begin{tikzpicture}[scale=0.3,baseline={([yshift=-.5ex]current bounding box.center)}]
\def\cz{5}
\def\wi{0.5}

\newcommand{\ci}[1]{	
	\fill[black] (#1) circle (\cz pt);
	\draw (#1) circle (\cz pt);
}

\coordinate (R) at (0,0);
\coordinate (r1) at (0,-1);

\draw (R) to (r1);

\ci{R};
\ci{r1};

\end{tikzpicture}(w) = 2 \,\,\begin{tikzpicture}[scale=0.3,baseline={([yshift=-.5ex]current bounding box.center)}]
\def\cz{5}
\def\wi{0.5}

\newcommand{\ci}[1]{	
	\fill[black] (#1) circle (\cz pt);
	\draw (#1) circle (\cz pt);
}

\coordinate (R) at (0,0);
\coordinate (l1) at (-\wi,-1);
\coordinate (r1) at (\wi,-1);
\coordinate (l11) at (-\wi,-2);

\draw (R) to (l1);
\draw (R) to (r1);
\draw (l1) to (l11);

\ci{R}
\ci{l1}
\ci{r1}
\ci{l11}

\end{tikzpicture}(w) + \begin{tikzpicture}[scale=0.3,baseline={([yshift=-.5ex]current bounding box.center)}]
\def\cz{5}
\def\wi{0.5}

\newcommand{\ci}[1]{	
	\fill[black] (#1) circle (\cz pt);
	\draw (#1) circle (\cz pt);
}

\coordinate (R) at (0,0);
\coordinate (r1) at (\wi,-1);
\coordinate (l1) at (-\wi,-1);

\draw (R) to (l1);
\draw (R) to (r1);

\ci{R};
\ci{l1};
\ci{r1};
\end{tikzpicture}\, \begin{tikzpicture}[scale=0.3,baseline={([yshift=-.5ex]current bounding box.center)}]
\def\cz{5}
\def\wi{0.5}

\newcommand{\ci}[1]{	
	\fill[black] (#1) circle (\cz pt);
	\draw (#1) circle (\cz pt);
}

\coordinate (R) at (0,0);
\ci{R}
\end{tikzpicture} (w) - \begin{tikzpicture}[scale=0.3,baseline={([yshift=-.5ex]current bounding box.center)}]
\def\cz{5}
\def\wi{0.5}

\newcommand{\ci}[1]{	
	\fill[black] (#1) circle (\cz pt);
	\draw (#1) circle (\cz pt);
}

\coordinate (R) at (0,0);
\coordinate (a1) at (-\wi,-1);
\coordinate (a2) at (0,-1);
\coordinate (a3) at (\wi,-1);

\draw (R) to (a1);
\draw (R) to (a2);
\draw (R) to (a3);

\ci{R}
\ci{a1}
\ci{a2}
\ci{a3}

\end{tikzpicture}(w) - \begin{tikzpicture}[scale=0.3,baseline={([yshift=-.5ex]current bounding box.center)}]
\def\cz{5}
\def\wi{0.5}

\newcommand{\ci}[1]{	
	\fill[black] (#1) circle (\cz pt);
	\draw (#1) circle (\cz pt);
}

\coordinate (R) at (0,0);

\coordinate (l1) at (0,-1);
\coordinate (l11) at (\wi,-2);
\coordinate (l12) at (-\wi,-2);

\draw (R) to (l1);

\draw (l1) to (l11);
\draw (l1) to (l12);

\ci{R}
\ci{l1}
\ci{l11}
\ci{l12}
\end{tikzpicture}(w) \,. \]
\end{remark}

\end{document}